\documentclass[12pt]{amsart}

\usepackage{amsmath,amsfonts,amssymb,amscd,amsthm,amsbsy,graphics,epsf}
\usepackage{comment,graphicx}

\theoremstyle{plain}
\newtheorem{theorem}{Theorem}[section]
\newtheorem{corollary}[theorem]{Corollary}

\newtheorem{lemma}[theorem]{Lemma}
\newtheorem{proposition}[theorem]{Proposition}
\newtheorem{example}[theorem]{Example}
\newtheorem{remark}[theorem]{Remark}
\newtheorem{definition}[theorem]{Definition}

\newtheorem{question}[theorem]{Question}

\numberwithin{equation}{section}

\makeatletter               
\let\c@equation\c@theorem       
\makeatother                

\begin{document}

\title[Finite ball intersection property of the Urysohn universal space]{Finite ball intersection property of the Urysohn universal space}

\author[A. Aksoy]{Asuman G\"{u}ven Aksoy}


\author[Z. Ibragimov]{Zair Ibragimov}


\date{December 26, 2013}

\thanks{}



\begin{abstract}
In a paper published posthumously, P.S. Urysohn constructed a complete, separable metric space that contains an isometric copy of every complete separable metric space, nowadays referred to as the Urysohn universal space.
Here we study various convexity properties of the Urysohn universal space and show that it has a finite ball intersection property. We also note that Urysohn universal space is not hyperconvex.
\end{abstract}

\maketitle

\section{ Introduction}

It is well known that both $l^{\infty}$ and $\mathcal{C}[0,1]$ are universal spaces. Indeed, every separable metric space isometrically embeds in a Banach space $l^{\infty}$ (Fr$\acute{e}$chet embedding) and a theorem of Banach \cite{Banach} states that every separable metric space embeds isometrically in $\mathcal{C}[0,1]$. Here $\mathcal{C}[0,1]$ is the separable Banach space of continuous real-valued functions on the closed unit interval equipped with the sup norm. However, the interest in the Urysohn space, denoted by $\mathbb{U}$, does not lie in its universality alone; it has the following finite transitivity property: every isometry between finite subsets of $\mathbb{U}$ extends to an isometry of $\mathbb{U}$ onto itself. More precisely, Urysohn proved the following theorem.

\begin{theorem}[\cite{Urysohn}]
Let $X$ be a separable and complete metric space that contains an isometric image of every separable metric space. Then $X$ is Urysohn universal if and only if it has the finite transitivity property.
\end{theorem}

For this reason $\mathbb{U}$ is called \emph{Urysohn universal}, not just universal.
In \cite{Urysohn} Urysohn also proved that $\mathbb{U}$ is unique, up to an isometry. It is worth remarking that the Banach space $\mathcal{C}[0,1]$ cannot be Urysohn universal since every isometric bijection between Banach spaces is an affine map (see \cite{BL}, page 341).

In this paper we study various convexity properties of the Urysohn universal space $\mathbb{U}$ and show that it has a finite ball intersection property (Theorem~\ref{theorem 1}). We also note that Urysohn universal space is not hyperconvex (Remark~\ref{rem}). It is worth noting that Urysohn's ideas have been extensively explored in geometry and topology (see, for example, \cite{HN,Hus,Me,Us,Ver}).

\section{Urysohn's construction}

Urysohn's original construction of $\mathbb{U}$ is published in full details in \cite{Urysohn}. For the sake of completeness and accessibility we shall briefly go through the construction. An alternative description of the Urysohn universal space is given in \cite[p. 20]{Heinonen}.

Urysohn first constructs a countable metric space $\mathbb{U}_0$ containing the image of every countable metric space for which the distance between any two points is rational. The metric completion of $\mathbb{U}_0$, (i.e., the unique complete metric space that contain $\mathbb{U}_0$ as a dense subset) is then the Urysohn universal space $\mathbb{U}$. We now proceed to the construction of the space $\mathbb{U}_0$. We start with an arbitrary countable set $\mathbb{U}_0$=$\{ a_1,a_2, \cdots, a_n, \cdots \}$ and define an appropriate metric on it. To define such a metric, Urysohn first considers the collection of all nonempty finite subsets of positive rational numbers. Denote this collection by $\mathcal{Q}$ and enumerate it as follows. First, consider all the elements of $\mathcal{Q}$ that consist of only one rational number and enumerate them using all the natural numbers that are not divisible by $4$. Now for each $p>1$, consider all the elements of $\mathcal{Q}$ that consist of $p$ rational numbers and enumerate them using all the natural numbers divisible by $2^p$, but not divisible by $2^{p+1}$. In this way, every element of $\mathcal{Q}$ receives a unique label $\mathbb Q_n$ for some natural number $n$. Thus, we obtain
$$ \mathcal Q=\{\mathbb{Q}_1,  \mathbb{Q}_2, \cdots,  \mathbb{Q}_n, \cdots,\}.$$
For example, $\mathbb Q_1$, $\mathbb Q_2$ and $\mathbb Q_3$ are single rational numbers; $\mathbb Q_4$, $\mathbb Q_{12}$ and $\mathbb Q_{20}$ consist of two rational numbers; $\mathbb Q_8$, $\mathbb Q_{24}$ and $\mathbb Q_{40}$ consist of three rational numbers and so on. Hence each $\mathbb{Q}_n$ can be written in the form
$$ \mathbb{Q}_n = [ r_1^{(n)}, r_2^{(n)}, \cdots  ,r_{p_n}^{(n)}]$$
where $r_1^{(n)}, r_2^{(n)}, \cdots r_{p_n}^{(n)}$ are the rational elements of $\mathbb{Q}_n$.
It is clear that $p_1=1$ and $p_n<n$, where $p_n$ is the cardinality of $\mathbb{Q}_n$. The metric on $\mathbb{U}_0$ is defined in the following way. We begin by setting $\rho(a_1,a_1)=0$. Suppose for all $i,k < n+1$, the nonnegative value $\rho (a_i,a_k)$  are defined. For $i\leq p_n$, consider the following two cases:

{\it{Case 1.}} At least one of the inequalities
\begin{equation}\label{Ury}
 |r_i^{(n)}-r_k^{(n)}| \leq \rho(a_i,a_k) \leq r_i^{(n)}+r_k^{(n)} \,\,\, \mbox{where}\,\,\,i,k \leq p_n
\end{equation}
is not satisfied. Urysohn calls such $\mathbb{Q}_n$ to be \emph{incorrectly defined}. In this case, we define
$$ \rho(a_{n+1}, a_j)=\displaystyle \max_{i,k \leq n} \rho(a_i,a_k),$$ for all $j \leq n$.

{\it{Case 2.}} All of inequalities in (\ref{Ury}) are satisfied. Urysohn calls such $\mathbb{Q}_n$ to be \emph{correctly defined}. In this case, we define
$$\rho(a_{n+1}, a_j)= \displaystyle \min_{ \lambda \leq p_n} \{\rho(a_j, a_{\lambda} ) +r_{\lambda}^{(n)}\},\,\,\, \mbox{for all} \,\, j \leq n.$$

Urysohn shows that the distance function $\rho$ is indeed a metric and that $\mathbb{U}_0$ is universal space for all countable metric spaces having rational values for their metrics.

\begin{remark}
Although it is not clear from Urysohn's construction, the sets of the form $\{r,r,r,\dots,r\}$ and $\{r\}$ should be considered to be the same. Indeed, the following possibility illustrates that they must be considered the same. Let $\mathbb Q_1=\{2\}$, $\mathbb Q_2=\{3\}$, $\mathbb Q_3=\{4\}$ and $\mathbb Q_4=\{1/2, 1/2\}$. If such $\mathbb Q_4$ was allowed to be considered as having two elements, then according to Urysohn's definition $\mathbb Q_4$ would be incorrectly defined. In this case, $\rho(a_5,a_j)=2$ for all $j=1,2,3,4$ and $\rho(a_3,a_4)=7$, contradicting to the triangle inequality for $a_3,a_4,a_5$.
\end{remark}

In \cite{Urysohn} Urysohn proves many interesting properties of his space $(\mathbb U,\rho)$. The following theorem will be used throughout this paper and will be referred as the Fundamental Theorem of Urysohn (\cite[Theorem I]{Urysohn})
\begin{theorem}\label{ThmUry}
Given any finite subset $x_1,x_2, \cdots, x_n$ of $\mathbb{U}$ and any positive real numbers $\alpha_1, \alpha_2, \cdots ,\alpha_n$ satisfying $|\alpha_i-\alpha_j| \leq \rho(x_i,x_j) \leq \alpha_i+\alpha_j$ for all $i,j\leq n$, there exists $y\in \mathbb{U}$ such that $\rho(y,x_i) = \alpha_i$ for every $i=1,2, \cdots, n$.
\end{theorem}

\section{Convexity of the Urysohn universal space}

We begin with the definitions of various types of convexities of metric spaces. Throughout this section we denote by $B(x,r)$ the
{\it{closed}} ball centered at $x$ with radius $r$. Observe first that in any metric space $(X,d)$, the triangle inequality implies that if
$B(x,r)\cap B(y,s)\neq\emptyset$, then $d(x,y)\leq r+s$ for all $x,y\in X$ and $r,s>0$.
The space $(X,d)$ is said to be {\it{metrically convex}} provided that $d(x,y)\leq s+t$ implies $B(x,r)\cap B(y,s)\neq\emptyset$ for all $x,y\in X$ and $r,s>0$.

Clearly, if $X$ is a geodesic metric space, then it is metrically convex. Indeed, given $x,y\in X$ and $r,s>0$ with $d(x,y)\leq r+s$, let $\gamma\subset X$ be a geodesic joining $x$ and $y$. Since $s>0$, we have $d(x,y)<r$, so there exists $z\in\gamma$ such that $d(x,z)=r$. Hence $z\in B(x,r)$. On the other hand, since $d(x,y)=d(x,z)+d(z,y)$, we obtain $d(y,z)=d(x,y)-d(x,z)\leq r+s-r=s$ so that $z\in B(y,s)$. Thus, $B(x,r)\cap B(y,s)\neq\emptyset$.

As the Urysohn universal space is geodesic (\cite[Theorem V]{Urysohn}), we obtain the following lemma.

\begin{lemma}\label{Lemma 1}
The Urysohn universal space $(\mathbb{U},\rho)$ is metrically convex.
\end{lemma}

A stronger notion of convexity is the finite ball intersection property.

\begin{definition}
We say that a metric space $(X,d)$ satisfies the finite ball intersection property if for any finite points $x_1,x_2, \dots, x_n$ in $X$ and for any positive real numbers $r_1,r_2,\dots, r_n$ satisfying $d(x_i,x_j) \leq r_i+r_j$, we have $\displaystyle \bigcap_{i=1}^{n} B(x_i, r_i) \neq \emptyset$.
\end{definition}

In the context of metrically convex spaces the finite ball intersection property can be restated as follows.
\begin{remark}
A metrically convex space $X$ satisfies the finite ball intersection property if and only if for any finite collection of balls $B(x_1,r_1)$, $B(x_2,r_2)$,$\dots$, $B(x_n,r_n)$ we have
$$
\displaystyle \bigcap_{i=1}^{n} B(x_i, r_i) \neq \emptyset\quad\text{whenever}\quad B(x_i,r_i)\cap B(x_j,r_j) \neq \emptyset.
$$
\end{remark}

The above remark implies that in order to prove the finite ball intersection property of a metrically convex space it is sufficient to consider only those collections of balls $B(x_1,r_1)$, $B(x_2,r_2)$,$\dots$, $B(x_n,r_n)$ in which no ball is contained in another. That is, $B(x_i,r_i)\nsubseteq B(x_j,r_j)$ and $B(x_j,r_j)\nsubseteq B(x_i,r_i)$ for each $i\neq j$. One can easily observe that the condition that $B(x_i,r_i)\nsubseteq B(x_j,r_j)$ and $B(x_j,r_j)\nsubseteq B(x_i,r_i)$ implies $|r_i-r_j|\leq d(x_i,x_j)$. Indeed, assuming that $|r_i-r_j|> d(x_i,x_j)$, we obtain either $r_i>d(x_i,x_j) +r_j$ or $r_j>d(x_i,x_j)+r_i$. Due to symmetry we assume that $r_i>d(x_i,x_j)+r_j$. Then for any $y\in B(x_j,r_j)$ we have
$d(y, x_i) \leq d(y,x_j)+d(x_j,x_i)< r_j+r_i-r_j= r_i$ implying $y\in B(x_i,r_i)$.
Hence $B(x_j,r_j)\subset B(x_i,r_i)$, which is the required contradiction.

We conclude that in order to prove that a metrically convex space $(X,d)$ satisfies the finite ball intersection property it is enough to show that for any finite points $x_1,x_2, \dots, x_n$ in $X$ and for any positive real numbers $r_1,r_2,\dots, r_n$ satisfying $|r_i-r_j|\leq d(x_i,x_j) \leq r_i+r_j$, we have $\displaystyle \bigcap_{i=1}^{n} B(x_i, r_i) \neq \emptyset$. Since the Urysohn universal space $\mathbb U$ is metrically convex (Lemma~\ref{Lemma 1}), the following theorem is an easy consequence of Urysohn's fundamental theorem (Theorem~\ref{ThmUry}).

\begin{theorem}\label{theorem 1}
The Urysohn universal space $(\mathbb{U},\rho)$ satisfies the finite ball intersection property.
\end{theorem}

The fundamental theorem of Urysohn, mentioned above, implies that the space $\mathbb U$ has a stronger form of finite ball intersection property. Namely, if $x_1,x_2, \dots, x_n\in\mathbb U$ and $r_1,r_2,\dots, r_n\in (0,+\infty)$ satisfy $|r_i-r_j|\leq d(x_i,x_j) \leq r_i+r_j$, then there exists $y\in\mathbb U$ such that $\rho(y,x_i)=r_i$. That is, $\displaystyle \bigcap_{i=1}^{n} S(x_i, r_i) \neq \emptyset$, where $S(x,r)=\{y\in\mathbb U\colon \rho(x,y)=r\}$ denotes a sphere centered at $x$ and of radius $r$. In this respect the following question arises naturally:

\begin{question}
Suppose that $(X,d)$ is a metrically convex space. Under what conditions on $X$
$$
\text{does}\qquad\displaystyle \bigcap_{i=1}^{n} B(x_i, r_i) \neq \emptyset\qquad\text{imply}\qquad\displaystyle \bigcap_{i=1}^{n} S(x_i, r_i) \neq \emptyset
$$
whenever $n\geq 2$ and $|r_i-r_j|\leq d(x_i,x_j) \leq r_i+r_j$?
\end{question}

\section{Hyperconvexity of metric spaces}

A stronger form of the finite ball intersection property is the notion of hyperconvexity.

\begin{definition}\label{definition}
A metric space $(X,d)$ is said to be hyperconvex if
$$
\bigcap_{i\in I } B(x_{i},r_{i})\neq\emptyset
$$
for every collection $B(x_{i},r_{i})$ of balls in $X$ for which $d(x_{i},x_{j})\leq r_{i}+r_{j}$.
\end{definition}

The notion of hyperconvexity was first introduced by Aronszajn and Panitchpakdi in \cite{ap}, where it was shown that a metric space is hyperconvex
if and only if it is injective with respect to nonexpansive ($1$-Lipschitz) mappings. Later Isbell \cite{Isbell} showed that every metric space has
an injective hull, which is the minimal hyperconvex space containing the given space as an isometric subspace. Hyperconvex metric spaces are complete and connected \cite{Akso}. The simplest examples of hyperconvex spaces are the set of real numbers $\mathbb R$, or a finite-dimensional real Banach space endowed with the maximum norm. While the Hilbert space $l_{2}$ fails to be hyperconvex, the spaces $L^\infty$ and $l^\infty$ are hyperconvex. In \cite{Akso} it is also shown that there is a general ``linking construction'' yielding hyperconvex spaces.

\begin{lemma}
The space $c_0$ is not hyperconvex.
\end{lemma}\begin{proof}
Recall that $c_0$ is the space of null sequences equipped with the metric $d$,
$$
d(x,y) = \operatorname{sup}_{n}|x_n-y_n|,\qquad\text{where}\qquad x=(x_n)\ \ \text{and}\ \ y=(y_n).
$$
Consider the canonical basis $e_n=(e_n^k)$ of $c_0$, where $e_n^k=0$ for $k\neq n$ and $e_n^k=1$ for $k=n$.
Let $B_n=B(e_n,1/2)$ be a ball centered at $e_n$ and of radius $1/2$. Then $d(e_i,e_j) =1 \leq 1/2+1/2$ so that the condition in the definition of hyperconvexity is satisfied. But intersection of these balls is empty. Indeed, if $x\in\displaystyle\bigcap_{k=1}^{\infty}B_k$, then $x\in B_n$ for some $n$, yielding $d(x, e_n) \leq 1/2$. But such a sequence $x$ cannot converge to $0$ and hance $x\notin c_0$. Thus, $\displaystyle\bigcap_{k=1}^{\infty}B_k=\emptyset$ in $c_0$ and, consequently, $c_0$ is not hyperconvex.
\end{proof}

\begin{lemma}
Suppose that $X$ is a separable metric space. If
$$
\bigcap_{i=1}^{\infty}B(x_i,r_i)\neq\emptyset
$$
for every countable collection $B(x_i,r_i)$ of balls in $X$ with $d(x_i,x_j)\leq r_i+r_j$, then $X$ is hyperconvex.
\end{lemma}\begin{proof}
Recall that a topological space is said to be Lindel\"of if every open cover has a countable subcover and that a metric space is Lindel\"of if and only if it is separable (see, for example, \cite[p. 192]{MUN}). Since $X$ is separable, every open cover of $X$ has a countable subcover.
Let $B(x_\alpha,r_\alpha)$ be an uncountable collection of balls in $X$ with $d(x_i,x_j)\leq r_i+r_j$. We need to show that $\displaystyle\bigcap_{\alpha\in\mathcal I} B(x_\alpha,r_\alpha)\neq\emptyset$.

Assume that $\displaystyle\bigcap_{\alpha\in\mathcal I} B(x_\alpha,r_\alpha)=\emptyset$. Then $X\setminus\displaystyle\bigcap_{\alpha\in\mathcal I} B(x_\alpha,r_\alpha) = X$. By the De Morgan formulas,
$$
X= X\setminus\bigcap_{\alpha\in\mathcal I} B(x_\alpha,r_\alpha)=\bigcup_{\alpha\in\mathcal I}\big(X\setminus B(x_\alpha,r_\alpha)\big).
$$
Since $X\setminus B(x_\alpha,r_\alpha)$ is open, the collection $\{X\setminus B(x_\alpha,r_\alpha), \alpha\in\mathcal I\}$ is an open cover of $X$. Since $X$ is Lindel\"of, there exists a countable subcover, say
$$
X\setminus B(x_1,r_1),\ \ X\setminus B(x_2,r_2),\ \ \dots, X\setminus B(x_n,r_n),\dots
$$
That is, $X=\displaystyle\bigcup_{k=1}^{\infty}\big(X\setminus B(x_k,r_k)\big)$. Again, by De Morgan we have
$$
X=\bigcup_{k=1}^{\infty}\big(X\setminus B(x_k,r_k)\big)= X\setminus\bigcap_{k=1}^{\infty} B(x_k,r_k)
$$
so that $\displaystyle\bigcap _{k=1}^{\infty} B(x_k,r_k)=\emptyset$. However, since $\rho(x_i,x_j)\leq r_i+r_j$, our hypothesis implies that $\displaystyle\bigcap _{k=1}^{\infty} B(x_k,r_k)\neq\emptyset$, which is a required contradiction.

\end{proof}

As a corollary we obtain
\begin{corollary}
Suppose that $(X,d)$ is a separable, metrically convex space. If
$$
\bigcap_{k=1}^{\infty}B(x_k, r_k)\neq\emptyset
$$
for any countable subset $x_1,x_2, \dots, x_n,\dots $ of $X$ and any positive real numbers $r_1, r_2, \dots ,r_n,\dots$, satisfying
$|r_i-r_j| \leq d(x_i,x_j) \leq r_i+r_j$ for all $i,j=1,2,\dots$, then $X$ is hyperconvex.

\end{corollary}


\section{Isbell's Hyperconvex Hull}

Isbell in \cite{Isbell} introduced injective envelope of a metric space, called the hyperconvex hull of the space. In what follows we give some of Isbell's ideas (see, also, \cite{Dress,DMT,Lang}). Let $(M,d)$ be a metric space. Given $x\in M$, let $f_x: M\rightarrow [0,\infty)$ be defined by $f_x(y)=d(x,y)$. Using the triangle inequality one can easily show (cf. \cite{EK}) that
$$ d(x,y) \leq f_a(x)+f_a(y) \,\,\,\,\mbox{and}\,\,\,\, f_a(x) \leq d(x,y)+f_a(y)$$ for any $x,y,a \in M$. Furthermore, if we let $f:M\rightarrow [0,\infty)$ be such that $d(x,y)\leq f(x)+f(y)$ for any $x,y \in M$, and if $ f(x) \leq f_a(x)$ for all $x\in M $ and for some $a\in M$, then $f_a= f$.

Let $A$ be any subset of $M$. We say that a function $f: A \rightarrow [0, \infty)$ is \emph{extremal} if $d(x,y)\leq f(x)+f(y)$ for all $x,y \in A$ and if $f$ is pointwise minimal. That is, if $g: A \rightarrow [0, \infty)$ is another function with the property that $d(x,y)\leq g(x)+g(y)$ for all $x,y \in A$ and $g(x) \leq f(x)$ for all $x\in A$, then we have $f=g.$
It is not difficult to see that every extremal function is non-negative and $1$-Lipschitz.

\begin{definition}
Let $A$ a nonempty subset of $M$. The injective envelope of $A$, denoted by $h(A)$, is the set of all extremal functions defined on $A$. In other words,
 $$
 h(A)=\{ f: A\rightarrow [0, \infty):\,\,d(x,y)\leq f(x)+f(y)\,\,\,\mbox{and}\,\, f \,\,\mbox{is pointwise minimal} \}
 $$
\end{definition}
The distance function $\rho(f,g)= \displaystyle \sup_{x\in A} d(f(x),g(x))$ defines a metric on $h(A)$ and
the map $e: A \rightarrow h(A)$, defined by $e(a)=f_a$, is an isometry. Indeed,
$$ \rho (e(a),e(b)) =\displaystyle \sup_{x\in A} |f_a(x)-f_b(x)|= \displaystyle \sup_{x\in A} |d(a,x)-d(b,x)|= d(a,b).$$
Thus, one can identify the subset $A$ with the subspace $e(A)$ of $h(A)$. Furthermore, we have the following extension property \cite{Isbell}.
Let $r :A \rightarrow [0,\infty)$ be such that $d(x,y) \leq r(x)+r(y)$ for all $x,y \in A$. Then there exist $R: M \rightarrow [0, \infty)$ which extends $r$ and such that $d(x,y) \leq R(x)+R(y)$ for all $x,y \in M$. Additionally, there exists an extremal function $f$ defined on $M$ such that $f(x) \leq R(x)$ for all $x \in M$. Aronszajn and Panitpacti proved the following two criteria of injectivity of metric spaces (\cite{ap}).

\begin{enumerate}
\renewcommand{\labelenumi}{\alph{enumi})}

\item A metric space $M$ is injective if and only if it is an absolute $1$-Lipschitz retract.
\item A metric space $M$ is injective if and only if it is hyperconvex.
\end{enumerate}

In the following we list some properties of $h(A)$ (cf. \cite{EK}).

\begin{enumerate}
\item if $f\in h(A)$, then it satisfies $|f(x)-f(y)| \leq d(x,y) \leq f(x)+f(y)$.
\item if $A$ is compact, then $h(A)$ is compact.
\item $h(A)$ is hyperconvex.
\item If $A \subset B \subset h(A)$, then $h(B)$ is isometric to $h(A)$.
\end{enumerate}

Following is a well known example which shows that hyperconvex hull of a set does not have to be unique.
\begin{example}
Consider the hyperconvex space $\mathbb{R}^2$ with the maximum norm and take $A= \{ (0,0),(0,1)\}$. Then the sets
$$ h_1(A)= \{ (x,y) \in \mathbb{R}^2: \,\,\, x=y, \, 0 \leq x\leq 1\,\}$$ and $$ h_2(A)= \{ (x,y) \in \mathbb{R}^2: \,\,\, x=y, \, 0 \leq x \leq 1/2\,\,\}\cup \{ (x,y) \in \mathbb{R}^2: \,\,\, x = 1-y, \, 1/2 \leq x \leq 1\,\,\,\}$$ are both hyperconvex hulls of $A$.
\end{example}
The following proposition is due to H. Herrlich (\cite[p. 187]{Herrlich}).
\begin{proposition}
$l^\infty$, the space of all real valued bounded sequences with the sup metric is the hyperconvex hull of its subspace $c_0$, consisting of all sequences converge to zero.
\end{proposition}

Herrlich's result implies that a hyperconvex hull of a separable metric space need not to be separable since $h(c_0)= l^{\infty}$.
Furthermore, Cohen proved the existence and uniqueness of a hyperconvex hull for any Banach or normed space over $\mathbb{R}$, and showed that an injective Banach space is linearly isometric to a function space $C(K)$, where $K$ is a compact Hausdorff and extremely disconnected (\cite{Cohen}). If $C(K)$ is separable, then $K$ is metrizable and hence discrete and first countable, thus even finite. In fact, $C(K)= l^{\infty}_n$ is finite dimensional, where $n=|K|$.

For more on hyperconvex hulls, see \cite{Dress} ,\cite{DMT} and \cite{Lang}. In particular, for hyperconvex hulls of normed spaces, we refer the reader to \cite{Rao}.
\begin{remark}\label{rem}
The results of Herrlich and Cohen mentioned above implies that the Urysohn universal space $\mathbb{U}$ is not hyperconvex.
\end{remark}
Indeed, if we assume that $\mathbb{U}$ is hyperconvex, then the separable infinite dimensional space $l^2$, which is not hyperconvex, embeds isometrically into $\mathbb{U}$. Hence $h(l^2)$ embeds isometrically in $h(\mathbb{U})=\mathbb{U}$. But since $h(l^2)$ is not separable, it can not isometrically embed into $\mathbb{U}$, a contradiction. The authors would like to thank U. Lang and N. Herzog for pointing out to us the result of Cohen and its implication on $\mathbb{U}$.

\bibliographystyle{amsplain}

\noindent
\mbox{~~~~~~~}Asuman G\"{u}ven Aksoy\\
\mbox{~~~~~~~}Claremont McKenna College\\
\mbox{~~~~~~~}Department of Mathematics\\
\mbox{~~~~~~~}Claremont, CA  91711, USA \\
\mbox{~~~~~~~}E-mail: aaksoy@cmc.edu \\ \\
\noindent
\mbox{~~~~~~~}Zair Ibragimov\\
\mbox{~~~~~~~}California State University, Fullerton\\
\mbox{~~~~~~~}Department of Mathematics\\
\mbox{~~~~~~~}Fullerton, CA, 92831, USA\\
\mbox{~~~~~~~}E-mail: zibragimov@fullerton.edu\\\\

\end{document}